\documentclass[a4paper,10pt]{amsart}
\usepackage[english]{babel}
\usepackage[latin1]{inputenc}
\usepackage{amsmath,amsfonts,amssymb,amsthm,amscd,array,stmaryrd,mathrsfs,bbm}
\usepackage[all]{xy}
\usepackage{anysize}\marginsize{25mm}{25mm}{30mm}{30mm}
\allowdisplaybreaks[4]
\usepackage[misc]{ifsym}
\usepackage{textcomp,booktabs}
\usepackage{graphicx}
\usepackage[usenames,dvipsnames]{color}
\usepackage{colortbl}  
\definecolor{mygray}{gray}{.7}
\usepackage[colorlinks,unicode=false,linkcolor=blue,citecolor=blue]{hyperref}

\numberwithin{equation}{section}
\theoremstyle{plain}
\newtheorem{theorem}{Theorem}

\newtheorem{lemma}{Lemma}[section]
\newtheorem{corollary}[lemma]{Corollary}

\theoremstyle{definition}
\newtheorem{definition}[lemma]{Definition}
\newtheorem{remark}[lemma]{Remark}
\newtheorem{example}[lemma]{Example}

\usepackage{tikz-cd}

\title{Sweedler Duality for BiHom-associative Algebras}

\date{\today}



\author{Jiacheng Sun}
\address{School of Mathematics, Southeast University, Nanjing 211189, China}
\email{220242018@seu.edu.cn}

\begin{document}

\begin{abstract}
Motivated by the fact that ordinary linear duality does not in general produce a coalgebra structure from an infinite-dimensional algebra, we develop a Sweedler-type finite dual construction for BiHom-associative algebras. 
For a BiHom-algebra $(G,\mu,\alpha,\beta)$ over a field, we define its Sweedler dual $G^{\circ}\subseteq G^{*}$ as the subspace of linear functionals annihilating a finite-codimensional BiHom-ideal of $G$. 
We prove that $G^{\circ}$ carries a natural BiHom-coalgebra structure whose comultiplication is the restriction of $\mu^{*}$, and that BiHom-algebra morphisms induce BiHom-coalgebra morphisms on Sweedler duals. 
We further extend this construction to right BiHom-modules, obtaining right BiHom-comodules over $G^{\circ}$ under a surjectivity assumption on the twisting map $\beta$. 
The Hom and classical cases are recovered by the specializations $\alpha=\beta$ and $\alpha=\beta=\mathrm{Id}$, respectively.
\end{abstract}

        \maketitle

        \section{Introduction}
	Algebraic deformation theory, initiated by Gerstenhaber~\cite{G}, has become a central theme in modern algebra. A different deformation paradigm is provided by \emph{Hom-type} structures, where classical identities are twisted by linear self-maps. Hom--Lie algebras were introduced by Hartwig, Larsson and Silvestrov~\cite{HLS} in connection with $q$-deformations of the Witt and Virasoro algebras, and related quasi--Lie and quasi--Hom--Lie structures were developed in~\cite{LS}. Makhlouf and Silvestrov subsequently studied several Hom-algebraic systems, including Hom-associative and Hom-Lie admissible algebras~\cite{MS1}, and established the corresponding (co)algebraic framework~\cite{MS2,MS3}. Further representation-theoretic and categorical aspects, such as Hom-(co)modules and Hom-Hopf-type constructions, were investigated by Yau~\cite{Y1,Y2,Y3,Y4}. The BiHom setting, involving two commuting twisting maps, has also attracted considerable attention and connects to matrix theory~\cite{SWZ}, integrable systems~\cite{GLSWZ}, braided representations~\cite{LWWZ}, and quantum-algebraic structures~\cite{Sheng,SD}.

Exponentials of infinitesimal characters and their convolution products give rise to characters with rich combinatorial interpretations; see~\cite{DBHPS}.
For a Hopf algebra, the characters form a group, and under some conditions this group carries a natural structure of an infinite-dimensional Lie group~\cite{Du}.
These considerations point to the \emph{Sweedler duality}, i.e.\ the largest dual space on which the induced comultiplication is well defined.
The same construction also appears in connections with automata with multiplicities and algebraic combinatorics~\cite{DFLL,V}.

In the finite-dimensional case, the usual linear dual of a Hom-algebra canonically yields a Hom-coalgebra, and the same holds for modules and comodules. 
In infinite dimension this naive dualization typically fails, since the transpose of the multiplication (or action) need not land in the appropriate tensor product.
The problem of finding a suitable dual construction in the Hom setting was raised by Wang.
In~\cite{SWZZ}, together with Wang, Zhang and Zhu, we addressed this question by adapting Sweedler's finite dual: for a Hom-algebra $G$ one considers a distinguished subspace $G^{\circ}\subseteq G^{*}$ defined via finite-codimensional ideals, which carries a natural Hom-coalgebra structure, and similarly for Hom-modules.
The purpose of the present paper is to extend these constructions to the BiHom framework, where two commuting twisting maps are involved.
We show that the Sweedler duality $G^{\circ}$ for a BiHom-algebra inherits a BiHom-coalgebra structure, and we develop the analogous module--comodule correspondence for right BiHom-modules.
The Hom case is recovered when $\alpha=\beta$, and the classical finite dual when $\alpha=\beta=\mathrm{Id}$.

The paper is organized as follows.
In Section~\ref{sect2} we recall basic notions on BiHom-(co)algebras and establish the BiHom-coalgebra structure on the Sweedler duality $G^{\circ}$ (Theorem~\ref{f}), together with functoriality on morphisms (Corollary~\ref{f1}).
Section~\ref{sect3} treats the module side: we construct the induced right $G^{\circ}$-BiHom-comodule structure on $M^{\circ}$ (Theorem~\ref{f2}) and show that BiHom-module morphisms dualize to BiHom-comodule morphisms (Theorem~\ref{f3}).

\paragraph{Notation.}
Throughout, $\mathbb K$ is a field of characteristic $0$.
For a vector space $V$ we write $\langle f,v\rangle := f(v)$ for $f\in V^{*}$ and $v\in V$.
For a BiHom-algebra $(G,\mu,\alpha,\beta)$ we consider the dual maps $
\mu^{*}:G^{*}\longrightarrow (G\otimes G)^{*}$, $\langle \mu^{*}(f),x\otimes y\rangle = \langle f,\mu(x\otimes y)\rangle$, and $
\alpha^{*}(f)=f\circ \alpha$, $\beta^{*}(f)=f\circ \beta$.
For a right BiHom-module $(M,\rho,\kappa,\tau)$ over $G$ we similarly set $\rho^{*}:M^{*}\longrightarrow (M\otimes G)^{*}$, $\langle \rho^{*}(\xi),m\otimes g\rangle = \langle \xi,\rho(m\otimes g)\rangle$,
together with $\kappa^{*}(\xi)=\xi\circ\kappa$ and $\tau^{*}(\xi)=\xi\circ\tau$.
Whenever one tensor factor is finite-dimensional, we use the canonical identifications
$(G\otimes G)^{*}\cong G^{*}\otimes G^{*}$ and $(M\otimes G)^{*}\cong M^{*}\otimes G^{*}$.
 
\section{Sweedler duality for BiHom-algebras}\label{sect2}

        In this section, we will first give an overview of the basic concepts related to the structures of BiHom-(co)algebras and the interrelationship between different BiHom-(co)algebras (see Definition 3.3 and Definition 5.1 in \cite{GMMP}). We then move on to introduce the definition of the {\em Sweedler duality}, which differs from the conventional duality (see Chapter 6 in \cite{G},\cite{SWZZ}).

        \begin{definition}
        A {\em BiHom-associative algebra} is a quadruple $(G, \mu, \alpha,\beta)$, where $G$ is a linear space, $\mu: G \otimes G \rightarrow G$ is a bilinear multiplication with $\mu(g\otimes g')=gg'$. For simplicity, when there is no confusion, we will omit the notation $\mu$. $\alpha,\beta: G \rightarrow G$ are two commutative homomorphisms, i.e., $\alpha\circ\beta=\beta\circ\alpha$. It is required that
       \begin{align}
       \alpha(g)(hk) &= (gh)\beta(k), \quad \text{ (BiHom-associativity)} \label{homAssoc}\\
       \alpha(hk)&=\alpha(h)\alpha(k). \quad \text{ ($\alpha$-multiplicativity)} \label{multiplicativity_1}\\
       \beta(hk)&=\beta(h)\beta(k). \quad \text{ ($\beta$-multiplicativity)} \label{multiplicativity_11}
       \end{align}
       \end{definition}
       
        They can be represented by the following commutative diagrams:
        \begin{equation*}
        \xymatrix{G\otimes G\otimes G\ar@{->}[d]_{\alpha\otimes \mu} & &G\otimes G \ar@{<-}_{\mu\otimes\beta}[ll]\ar@{->}_{\mu}[d] \\
        G\otimes G & & G, \ar@{<-}_{\mu }[ll] }
        \qquad 
        \xymatrix{G\otimes G\ar@{->}[d]_{\mu} & &G\otimes G \ar@{<-}_{\alpha\otimes\alpha}[ll]\ar@{->}_{\mu}[d] \\
        G & & G. \ar@{<-}_{\alpha }[ll] }
        \qquad
        \xymatrix{G\otimes G\ar@{->}[d]_{\mu} & &G\otimes G \ar@{<-}_{\beta\otimes\beta}[ll]\ar@{->}_{\mu}[d] \\
        G & & G. \ar@{<-}_{\beta }[ll] }
        \end{equation*}
        
        In the following of this paper, we will abbreviate {\em BiHom-associative algebra} to {\em BiHom-algebra}, and sometimes we refer to a BiHom-algebra $(G, \mu, \alpha,\beta)$ as $G$. Now, for two BiHom-algebras $(G, \mu, \alpha,\beta)$ and $(G', \mu', \alpha',\beta')$, let $f: G \to G'$ be a linear map, then $f$ is said to be a morphism of BiHom-algebras if it satisfies the following conditions
        \begin{align}\label{s1}
        \mu' \circ (f \otimes f) = f \circ \mu, \quad\quad f \circ \alpha = \alpha' \circ f, \quad\quad f \circ \beta = \beta' \circ f.
        \end{align}

       \begin{definition}
        Let $(G,\mu,\alpha,\beta)$ be a BiHom-algebra, then a subspace $H$ of $G$ is a {\em BiHom-associative subalgebra} of $G$ if 
        $bc\in H$, $\alpha(b)\in H$ and $\beta(b)\in H$ for any $b,c\in H$.
        In particular, if both $ab\in H$ and $ba\in H$ hold for any $a\in G$ and $b\in H$, we say that $H$ is an {\em ideal} of $G$.
       \end{definition}

       \begin{definition}
       A \textit{BiHom-coalgebra} is a quadruple $(C, \Delta, \psi,\phi)$, where $C$ is a linear space, $\Delta: C \rightarrow C \otimes C$ is a linear map with notation $\Delta(c):=\sum_{(c)}c_{(1)}\otimes c_{(2)}$ for all $c\in C$. And $\psi,\phi: C \rightarrow C$ are two commutative homomorphisms, i.e., $\psi\circ\phi=\phi\circ\psi$ which satisfies
      \begin{align}
      (\phi \otimes \Delta) \circ \Delta &= (\Delta \otimes \psi) \circ \Delta, \quad \text{ (BiHom-coassociativity)} \label{hom-coassoc} \\
      \Delta\circ\psi&=(\psi\otimes\psi)\circ\Delta,  \quad \text{ ($\psi$-comultiplicativity)}\label{multiplicativity_2}\\
      \Delta\circ\phi&=(\phi\otimes\phi)\circ\Delta.  \quad \text{ ($\phi$-comultiplicativity)}\label{multiplicativity_22}
       \end{align}
        \end{definition}

        They can be expressed by the following commutative diagrams:
\begin{equation*}
\xymatrix{C\otimes C\otimes C\ar@{<-}[d]_{\Delta\otimes \psi} & &C\otimes C \ar@{->}_{\phi\otimes \Delta}[ll]\ar@{<-}_{\Delta}[d] \\
C\otimes C & & C, \ar@{->}_{\Delta }[ll] }
\qquad
\xymatrix{C\otimes C\ar@{<-}[d]_{\Delta} & &C\otimes C \ar@{->}_{\psi\otimes \psi}[ll]\ar@{<-}_{\Delta}[d] \\
C & & C. \ar@{->}_{\psi }[ll] }
\qquad
\xymatrix{C\otimes C\ar@{<-}[d]_{\Delta} & &C\otimes C \ar@{->}_{\phi\otimes \phi}[ll]\ar@{<-}_{\Delta}[d] \\
C & & C. \ar@{->}_{\phi }[ll] }
\end{equation*}\par

        Similarly, we will denote a BiHom-coalgebra $(C, \Delta, \psi,\phi)$ by $C$. A morphism of BiHom-coalgebras is a linear map of the underlying $\mathbb K$-linear spaces that commutes with the twisting maps $\psi$ and $\phi$ and the comultiplication $\Delta$.

        We begin with some basic facts on finite-dimensional duality for BiHom-(co)algebras.

        \begin{lemma}\label{s}
        Let $(G,\mu,\alpha,\beta)$ be a finite-dimensional BiHom-algebra.
        Then $(G^{*},\mu^{*},\beta^{*},\alpha^{*})$ is a BiHom-coalgebra, where
        \[
        \mu^{*}:G^{*}\longrightarrow G^{*}\otimes G^{*},\qquad
        \langle \mu^{*}(f),x\otimes y\rangle=\langle f,\mu(x\otimes y)\rangle,
        \]
        and the twisting maps are given by
        \[
        \alpha^{*}(f)=f\circ\alpha,\qquad \beta^{*}(f)=f\circ\beta\qquad (f\in G^{*}).
        \]
        \end{lemma}

        \begin{proof}
        Since $\dim G<\infty$, we identify $(G\otimes G)^{*}\cong G^{*}\otimes G^{*}$, so $\mu^{*}$ indeed takes values in $G^{*}\otimes G^{*}$.
        The BiHom-coalgebra axioms for $(G^{*},\mu^{*},\beta^{*},\alpha^{*})$ follow by dualizing the BiHom-algebra identities for $(G,\mu,\alpha,\beta)$ and checking them against the natural pairing. The argument parallels the Hom case~\cite{SWZZ}. 
        \end{proof}

\begin{lemma}\label{qq}
Let $(G,\mu,\alpha,\beta)$ and $(G',\mu',\alpha',\beta')$ be finite-dimensional BiHom-algebras.
A linear map $f:G\to G'$ is a BiHom-algebra morphism if and only if
$f^{*}:G'^{*}\to G^{*}$ is a BiHom-coalgebra morphism.
\end{lemma}

        \begin{proof}
            Since $f$ is a morphism of BiHom-algebras, then for each $a,b\in G$ and $h\in G'^*$, we have
            \begin{align*}
                \mu^*\circ f^*(h)(a\otimes b)&=\langle \mu^*\circ f^*(h), a\otimes b\rangle \\
                &=\langle h, f\circ\mu(a\otimes b)\rangle\\
                &=\langle h, \mu'\circ (f\otimes f)(a\otimes b)\rangle \\
                &=\langle \mu'^*(h), (f\otimes f)(a\otimes b)\rangle\\
                &=\langle h_{(1)}\otimes h_{(2)}, f(a)\otimes f(b)\rangle\\
                &=\langle f^*(h_{(1)})\otimes f^*(h_{(2)}), a\otimes b\rangle\\
                &=\langle (f^*\otimes f^*)\circ\mu'^*(h), a\otimes b\rangle\\
                &=(f^*\otimes f^*)\circ\mu'^*(h)(a\otimes b).
            \end{align*}
          Here, we note $\mu'^*(h)=h_{(1)}\otimes h_{(2)}$ using the Sweedler's notation. Hence $\mu^*\circ f^*=(f^*\otimes f^*)\circ\mu'^*$. Also we have 
          \begin{align*}
              f^*\circ\alpha'^*(h)(a)&=\langle f^*\circ\alpha'^*(h), a\rangle\\
              &=\langle h, \alpha'\circ f(a)\rangle\\
              &=\langle h, f\circ\alpha(a)\rangle\\
              &=\langle \alpha^*\circ f^*(h) ,a\rangle\\
              &=\alpha^*\circ f^*(h)(a).
          \end{align*}
          Hence $f^*\circ\alpha'^*=\alpha^*\circ f^*$. Similarly, we obtain $f^*\circ\beta'^*=\beta^*\circ f^*$. Therefore, $f^*$ is a morphism of BiHom-coalgebras. Conversely, if $f^*$ is a morphism of BiHom-coalgebras, we use the similar proof to conclude that $f$ is a morphism of BiHom-algebras.
        \end{proof}

       \begin{lemma}[{\cite[Chapter~II]{EABC}}]\label{zz}
Let $A$ and $B$ be $\mathbb K$-vector spaces, and let $I\subseteq A$, $J\subseteq B$ be subspaces.
Let $\pi_A:A\to A/I$ and $\pi_B:B\to B/J$ be the quotient maps, and consider the induced canonical surjection
\[
\pi:=\pi_A\otimes \pi_B:\; A\otimes B \longrightarrow (A/I)\otimes (B/J).
\]
Then
\[
\ker(\pi)=A\otimes J \;+\; I\otimes B,
\]
and hence there is a natural isomorphism
\[
(A\otimes B)/(A\otimes J+I\otimes B)\;\cong\; (A/I)\otimes (B/J).
\]
\end{lemma}

\begin{lemma}\label{M}
Let $(G,\mu,\alpha,\beta)$ and $(G',\mu',\alpha',\beta')$ be BiHom-algebras, and let
$f:G\to G'$ be a morphism of BiHom-algebras.
Let $H\subseteq G$ be an ideal and let $\pi:G\to G/H$ be the canonical quotient morphism.
If $\ker(\pi)=H\subseteq \ker(f)$, then there exists a unique BiHom-algebra morphism
$\bar f:G/H\to G'$ such that $\bar f\circ \pi=f$.
\end{lemma}

\begin{proof}
Since $H=\ker(\pi)\subseteq \ker(f)$, the universal property of the quotient yields a unique linear map
$\bar f:G/H\to G'$ such that $\bar f\circ \pi=f$.

To show that $\bar f$ is a BiHom-algebra morphism, it suffices to use that $\pi$ is surjective.
For $\bar x,\bar y\in G/H$ pick $x,y\in G$ with $\pi(x)=\bar x$ and $\pi(y)=\bar y$. Then
\[
\bar f\bigl(\mu_{G/H}(\bar x\otimes \bar y)\bigr)
=\bar f\bigl(\pi(\mu(x\otimes y))\bigr)
=f(\mu(x\otimes y))
=\mu'(f(x)\otimes f(y))
=\mu'\bigl(\bar f(\bar x)\otimes \bar f(\bar y)\bigr),
\]
so $\bar f\circ \mu_{G/H}=\mu'\circ(\bar f\otimes \bar f)$.
Similarly, using $f\circ\alpha=\alpha'\circ f$ and $f\circ\beta=\beta'\circ f$, we obtain
$\bar f\circ \alpha_{G/H}=\alpha'\circ \bar f$ and $\bar f\circ \beta_{G/H}=\beta'\circ \bar f$.
\end{proof}

\begin{remark}
By a \emph{regular full homomorphism} we mean the canonical quotient map
$\pi:G\to G/I$ associated with an ideal $I\subseteq G$; equivalently, $\pi$ is a surjective morphism of BiHom-algebras with $\ker(\pi)=I$.
\end{remark}

To extend the duality between BiHom-algebras and BiHom-coalgebras beyond the finite-dimensional setting, we now introduce the Sweedler dual $G^{\circ}$ and record a basic functoriality property for finite-codimensional ideals (Definition~\ref{Sweedler_duality_Hom_alg} and Lemma~\ref{b}).

\begin{definition}\label{Sweedler_duality_Hom_alg}
Let $G$ be a BiHom-algebra over $\mathbb K$.
The \emph{Sweedler duality} (or \emph{finite duality}) for $G$ is the subspace
\[
G^{\circ}
:=\bigl\{\, f\in G^{*}\ \big|\ \exists\ \text{a finite-codimensional ideal } J\subseteq G
\text{ such that } f(J)=0 \,\bigr\}.
\]
\end{definition}

\begin{lemma}\label{b}
    Let $(G,\mu,\alpha,\beta)$ and $(G',\mu',\alpha',\beta')$ be two BiHom-algebras, $f: G\rightarrow G'$ be a morphism of BiHom-algebras. If $J$ is a finite-codimensional ideal of $G'$, then its complete inversion $f^{-1}(J)$ is also a finite-codimensional ideal of $G$.
\end{lemma}

\begin{proof}
    According to equation~\eqref{s1}, we obtain the following, which holds for any \( x \in f^{-1}(J) \) and \( a \in G \):
    \begin{align*}
        f(x)\in J \quad \text{and} \quad f(ax)=f(a)f(x)\in J.
    \end{align*}
    Then \( ax \in f^{-1}(J) \). Similarly, we can show that \( xa \in f^{-1}(J) \).  Now, observe that $f$ is a morphism of BiHom-algebras. Consequently, we have
     $$
     f\circ\alpha(x) = \alpha^{\prime}\circ f(x) \in J,\text{ and } f\circ\beta(x) = \beta^{\prime}\circ f(x) \in J,
     $$
     there by establishing $\alpha(x) \in f^{-1}(J)$ and $\beta(x) \in f^{-1}(J)$. We then proceed with $\alpha(f^{-1}(J)) \subseteq f^{-1}(J)$ and $\beta(f^{-1}(J)) \subseteq f^{-1}(J)$. So $f^{-1}(J)$ is the ideal of BiHom-algebra $G$.

    We now demonstrate that $f^{-1}(J)$ is finite-codimensional. Let $g:G\to G'/J$ be a map defined as follows: for any $a\in G$, $a=a_1+a_2$ where $a_1\in f^{-1}(J)$, we have $g(a)=f(a_2)+J$. It is evident that $\operatorname{Ker}(g) = f^{-1}(J)$. Then we can induce the injection $\widetilde{g}:G/f^{-1}(J)\rightarrow G^{'}/J$ ($b+f^{-1}(J)\mapsto f(b)+J$), so we have
    \begin{align*}
       \dim(G/f^{-1}(J))\leq  \dim(G^{'}/J)<+\infty
    \end{align*}
    which implies that the dimension of $G/f^{-1}(J)$ is finite. Consequently, $f^{-1}(J)$ is a finite-codimensional ideal of $G$.
\end{proof}

Motivated by Liu's construction in the classical setting~\cite{LGL}, we now build a BiHom-coalgebra structure under the Sweedler duality.

\begin{theorem}\label{c}
    Let $G$ be a BiHom-algebra. Then $G^{\circ}$ is the subspace of $G^{*}$.
\end{theorem}

   \begin{proof}
   It is straightforward to verify that $G^{\circ}$ is closed under $\mathbb K$-multiplication. We proceed to show that $G^{\circ}$ is also closed under addition.

   For any $g,h\in G^{\circ}$, let $J$ and $H$ be two finite-codimensional ideals of $G$ such that $\operatorname{Ker} g\supseteq J$ and $\operatorname{Ker} h\supseteq H$. It is clear that $J\cap H$ is an ideal of the Hom-algebra and $\operatorname{Ker}
    (g+h)\supseteq J\cap H$.

   Our next step is to explain $J\cap H$ is finite-codimensional. Since $H/(J\cap H) \cong (J+H)/J \subseteq G/J$, the quotient $H/(J\cap H)$ is a finite-dimensional linear space.
   Furthermore, since $G/H \cong (G/(J \cap H))/(H/(J \cap H))$, the finite dimension of $G/H$ implies that $dim(G/(J \cap H)) = dim(G/H) + dim(H/(J \cap H))$ is also finite. Consequently, $J \cap H$ is a finite-codimensional ideal of $G$.
   \end{proof}

   \begin{theorem}\label{f}
   Let $(G,\mu,\alpha,\beta)$ be a BiHom-algebra. Then the Sweedler duality of $G$ is a BiHom-coalgebra $(G^{\circ},\Delta,\beta^{\circ},\alpha^{\circ})$, where $\Delta=\mu^{*}|_{G^{\circ}}$, $\alpha^{\circ}=\alpha^{*}|_{G^{\circ}}$, $\beta^{\circ}=\beta^{*}|_{G^{\circ}}$ and $\mu^{*}: G^{*}\rightarrow (G\otimes G)^*$.
   \end{theorem}

    \begin{proof}
    Let $f \in G^{\circ}$, we want to show that $\Delta(f) \in G^{\circ} \otimes G^{\circ}$. There exists a finite-codimensional ideal $J$ of $G$ such that $f(J) = 0$. Then we do the formal arithmetic
    \begin{align*}
        \Delta(f)(J \otimes G + G \otimes J) = f\circ\mu(J \otimes G + G \otimes J) = 0,
    \end{align*}
    since $\mu(J \otimes G)\subseteq J$ and $\mu(G \otimes J)\subseteq J$. Let a map $\pi\colon G \otimes G \to G \otimes G / (J \otimes G + G \otimes J)$, then we have $\operatorname{Ker}(\pi) = J \otimes G + G \otimes J \subseteq \operatorname{Ker}(f \mu)$. By Lemma~\ref{M}, this induces a unique morphism $\overline{f \mu}\colon G \otimes G / (J \otimes G + G \otimes J) \to \mathbb K$ satisfying $\overline{f \mu} \circ \pi = f\circ \mu$. That is to say $\overline{f \mu} \in (G \otimes G / (J \otimes G + G \otimes J))^*$.

    According to Lemma~\ref{zz}, we have $G \otimes G / (J \otimes G + G \otimes J) \cong (G/J) \otimes (G/J)$ and $G/J$ is finite, it is easy to see that
    \begin{align*}
        \overline{f\mu} \in(G/J)^* \otimes (G/J)^*.
    \end{align*}

    Let $\overline{f\mu} = \sum_{i,j=1}^n k_{ij} \overline{e_i}^* \otimes \overline{e_j}^*$, where $\overline{e_i}^* \in (G/J)^*$ is the duality base elements corresponding to a canonical map $\Tilde{\pi}\colon G \to G/J$. So for any $a,b \in G$, we have
    \begin{align*}
    \langle f\circ\mu, a \otimes b \rangle &= \langle \overline{f\mu} \circ \pi, a \otimes b \rangle \\
    &= \langle \overline{f\mu} , \pi(a \otimes b) \rangle \\
    &= \langle \overline{f\mu}, \Tilde{\pi}(a) \otimes \Tilde{\pi}(b) \rangle \\
    &= \sum_{i,j=1}^n k_{ij}\langle \overline{e}_i^*, \Tilde{\pi}(a) \rangle \langle \overline{e}_j^*, \Tilde{\pi}(b) \rangle.
    \end{align*}

    Denote $e_i^* = \overline{e}_i^* \circ \Tilde{\pi}$. Because of $\operatorname{Ker} (e_i^*) \supseteq J$, we have $e_i^* \in G^{\circ}$. It is straightforward that
    \begin{align*}
        \Delta(f) = f\circ \mu = \sum\limits_{i,j=1}^n k_{ij} e_i^* \otimes e_j^* .
    \end{align*}
    Therefore, we conclude that $\Delta(f) \in G^{\circ} \otimes G^{\circ}$.

    For any $g \in G^{\circ}$, there exists a finite-codimensional ideal $J$ of $G$ such that $\operatorname{Ker}(g) \supseteq J$. Then, since $\alpha(J)\subseteq J$ and $g(J)=0$, we have
    \begin{align*}
        \langle \alpha^{\circ}(g), J \rangle =\langle \alpha^{*}(g), J \rangle= \langle g, \alpha(J) \rangle= 0.
    \end{align*}
    Hence $\operatorname{Ker}(\alpha^{\circ}(g)) \supseteq J$ follows, which implies that $\alpha^{\circ}(g)\in G^\circ$. Similarly, $\beta^{\circ}(g)\in G^\circ$ holds.
    
    Using the equation \eqref{multiplicativity_1}, for any $c^*\in G^{\circ}$ and $x,y\in G$, we have
    \begin{align*}
        \langle \Delta\circ\alpha^{\circ}(c^*), x\otimes y\rangle&=\langle c^*, \alpha\circ\mu(x\otimes y)\rangle\\
        &=\langle c^*, \mu\circ(\alpha\otimes\alpha)(x\otimes y)\rangle\\
        &=\langle\Delta(c^*), \alpha(x)\otimes \alpha(y)\rangle\\
        &=\sum_{(c^*)}\langle c^*_{(1)}\otimes c^*_{(2)}, \alpha(x)\otimes \alpha(y)\rangle\\
        &=\sum_{(c^*)}\langle \alpha^{\circ}(c^*_{(1)}), x\rangle\langle \alpha^{\circ}(c^*_{(2)}), y\rangle\\
        &=\langle (\alpha^{\circ}\otimes \alpha^{\circ})(\sum_{(c^*)} c^*_{(1)}\otimes c^*_{(2)}), x\otimes y\rangle\\
        &=\langle (\alpha^{\circ}\otimes \alpha^{\circ})\circ\Delta(c^*), x\otimes y\rangle.
    \end{align*}
    Equivalently speaking, we have $\Delta\circ\alpha^{\circ}=(\alpha^{\circ}\otimes \alpha^{\circ})\circ\Delta$. Similarly, we obtain $\Delta\circ\beta^{\circ}=(\beta^{\circ}\otimes \beta^{\circ})\circ\Delta$.\par
    
    It is sufficient to verify that
    $(\alpha^{\circ}\otimes \Delta)\circ\Delta=(\Delta \otimes \beta^{\circ})\circ\Delta.$
     For any $c^{*}\in G^{\circ}$ and $x,y,z\in G$, we have
    \begin{align*}
  \langle (\alpha^{\circ} \otimes \Delta) \circ \Delta(c^*), x \otimes y \otimes z \rangle
  &= \langle (\alpha^{\circ} \otimes \Delta) \circ (\sum_{(c^*)} c^*_{(1)} \otimes c^*_{(2)}) , x \otimes y \otimes z \rangle \\
  &= \sum_{(c^*)} \langle \alpha^{\circ}(c^*_{(1)}) \otimes \Delta(c^*_{(2)}), x \otimes y \otimes z \rangle \\
  &=\sum_{(c^*)} \langle \alpha^{\circ}(c^*_{(1)}), x\rangle\langle\Delta(c^*_{(2)}),  y \otimes z \rangle \\
  &= \sum_{(c^*)} \langle c^*_{(1)}, \alpha(x) \rangle \langle c^*_{(2)}, yz \rangle \\
  &=\sum_{(c^*)} \langle c^*_{(1)}\otimes c^*_{(2)}, \alpha(x)\otimes yz\rangle\\
  &=\langle \Delta(c^*), \alpha(x)\otimes yz\rangle\\
  &= \langle c^*, \alpha(x) yz \rangle.
\end{align*}
    Similarly,
    \begin{align*}
  \langle (\Delta \otimes \beta^{\circ}) \circ \Delta(c^*), x \otimes y \otimes z \rangle &= \langle c^*, x y \beta(z) \rangle.
\end{align*}

    Due to the BiHom-associativity (\ref{homAssoc}) of $G$, we have $\alpha(x)yz=xy\beta(z)$. Thus $(\alpha^{\circ}\otimes \Delta)\circ\Delta=(\Delta \otimes \beta^{\circ})\circ\Delta$. 
\end{proof}

Without specific instructions, we will examine the properties of the BiHom-coalgebra $(G^{\circ},\Delta,\beta^{\circ},\alpha^{\circ})$, Sweedler duality of $G$ as described in Theorem~\ref{f}. Based on these considerations, we then derive the morphism of BiHom-coalgebras accordingly.

\begin{corollary}\label{f1}
    Let $(G,\mu,\alpha,\beta)$ and $(G',\mu',\alpha',\beta')$ be two BiHom-algebras, $(G^{\circ},\Delta,\beta^{\circ},\alpha^{\circ})$ and $(G'^{\circ},\Delta',\beta'^{\circ},\alpha'^{\circ})$ two corresponding objects under the Sweedler duality. The map $f: G\rightarrow G'$ is a morphism of BiHom-algebras. Note that the map $f^{\circ} = f^{*}|_{G'^{\circ}}:G'^{\circ}\rightarrow G^{*}$. Thus, $f^{\circ}$ is a morphism of BiHom-coalgebras.
\end{corollary}

\begin{proof}
It suffices to prove that $f^{\circ}(G'^{\circ})\subseteq G^{\circ}$.
Let $b^{*}\in G'^{\circ}$, and choose a finite-codimensional ideal $J\subseteq G'$ with $b^{*}(J)=0$.
By Lemma~\ref{b}, the inverse image $f^{-1}(J)$ is a finite-codimensional ideal of the BiHom-algebra $G$.

    Notice that
    \begin{align*}
        \langle f^{\circ}(b^{*}), f^{-1}(J) \rangle  = \langle b^{*}, f\circ f^{-1}(J) \rangle 
                          = \langle b^{*}, J \rangle 
                          = 0,
    \end{align*}
    which shows that $\operatorname{Ker}(f^{\circ}(b^{*})) \supseteq f^{-1}(J)$. In other words, $f^{\circ}(b^{*}) \in G^{\circ}$. Consequently, $f^{\circ}(G'^{\circ}) \subseteq G^{\circ}$.

    Using the equation (\ref{s1}), for any $b^{*} \in G^{\prime \circ}$ and $a, c \in G^{\circ}$, we have
    \begin{align*}
        \langle \Delta\circ f^{\circ}(b^{*}), a \otimes c \rangle &= \langle b^{*}, f\circ \mu(a \otimes c) \rangle \\
        &= \langle b^{*}, \mu^{\prime}\circ(f \otimes f)(a \otimes c) \rangle \\
        &= \langle  \Delta'(b^{*}), f(a) \otimes f(c) \rangle \\
        &=\sum_{(b^*)}\langle  b^*_{(1)}\otimes  b^*_{(2)}, f(a) \otimes f(c) \rangle \\
        &=\sum_{(b^*)}\langle  b^*_{(1)}, f(a)\rangle\langle  b^*_{(2)}, f(c) \rangle \\
        &=\sum_{(b^*)}\langle  f^{\circ} (b^*_{(1)}), a\rangle\langle  f^{\circ}(b^*_{(2)}), c \rangle \\
        &=\langle (f^{\circ} \otimes f^{\circ})(\sum_{(b^*)} b^*_{(1)}\otimes  b^*_{(2)}), a \otimes c \rangle\\
        &= \langle (f^{\circ} \otimes f^{\circ})\circ \Delta'(b^{*}), a \otimes c \rangle.
    \end{align*}
It means we obtain that $\Delta\circ f^{\circ} = (f^{\circ} \otimes f^{\circ})\circ \Delta'$.

    For any $b^* \in G'^{\circ}$ and $a \in G$, we have
\begin{align*}
    \langle \alpha^{\circ}\circ f^{\circ}(b^*), a \rangle = \langle b^*, f\circ \alpha(a) \rangle = \langle b^*, \alpha'\circ f(a) \rangle
    = \langle f^{\circ}\circ \alpha'^{\circ}(b^*), a \rangle.
\end{align*}
Therefore, it follows that $\alpha^{\circ}\circ f^{\circ} = f^{\circ}\circ \alpha'^{\circ}$. Similarly, we obtain $\beta^{\circ}\circ f^{\circ} = f^{\circ}\circ \beta'^{\circ}$. Therefore the desired result holds.
\end{proof}

Using the Sweedler duality that we constructed, ordinary BiHom-associative algebras naturally get an infinite-dimensional correspondence. We construct this structure over the Multivariate Polynomial in the following.

\begin{example}[Multivariate polynomial BiHom-algebras]
Let $\mathbb K$ be a field and we consider the BiHom-algebra $(\mathbb K[x_1,x_2,\cdots,x_r],\cdot,\alpha,\beta)$ where $\mathbb K[x_1,x_2,\cdots,x_r]$ is the set of the polynomials in $r$ variables, $\alpha$ and $\beta$ are both the linear maps acting on the indeterminate $x$ with the property $\alpha(\kappa x^{\mathbf m})=\kappa\prod_{k=1}^r(\sum_{l=1}^r a_{lk}x_l)^{m_k}$ and $\beta(\tau x^{\mathbf n})=\tau\prod_{k=1}^r(\sum_{l=1}^r b_{lk}x_l)^{n_k}$ for each coefficients $\kappa,\tau\in\mathbb K$, multi-indices $\mathbf m=(m_1,m_2,\cdots, m_r)$ and $\mathbf n=(n_1,n_2,\cdots, n_r)$ and the corresponding monomials $x^{\mathbf m}=x_1^{m_1}x_2^{m_2}\cdots x_r^{m_r}$ and $x^{\mathbf n}=x_1^{n_1}x_2^{n_2}\cdots x_r^{n_r}$, $A=(a_{ij})_{r\times r}\in \mathbb K^{r\times r}$ and $B=(b_{ij})_{r\times r}\in \mathbb K^{r\times r}$, satisfying $AB=BA$, which implies $\alpha\circ\beta=\beta\circ\alpha$. 
Then the bilinear operator $\cdot:\mathbb K[x_1,x_2,\cdots,x_r]\rightarrow \mathbb K[x_1,x_2,\cdots,x_r]$ is defined by 
\[
  x^{\mathbf m}\cdot x^{\mathbf n}=\alpha(x^{\mathbf m})\beta(x^{\mathbf n})=\Big(\sum_{\mathbf p}A_{\mathbf m}^{\mathbf{p}}x^{\mathbf p}\Big)\Big(\sum_{\mathbf q}B_{\mathbf n}^{\mathbf{q}}x^{\mathbf q}\Big)=\sum_{\mathbf{p},\mathbf{q}}A_{\mathbf m}^{\mathbf{p}}B_{\mathbf n}^{\mathbf{q}}x^{\mathbf p+\mathbf q}.
\]
Here, 
$$\sum_{\mathbf p}A_{\mathbf m}^{\mathbf{p}}x^{\mathbf p}=\prod_{k=1}^r(\sum_{l=1}^r a_{lk}x_l)^{m_k}=\sum_{\substack{
M=(m_{lk})_{r\times r}\in \mathbb N^{r\times r}, \text{ each }m_{lk}\geq 0,\\
\sum_{l=1}^r m_{lk}=m_k,\ \\[2pt]
\sum_{k=1}^r m_{lk}=p_l\ 
}}
\left(\prod_{k=1}^{r}\frac{m_k!}{\prod_{l=1}^{r} m_{lk}!}\right)
\left(\prod_{l,k=1}^{r} a_{lk}^{\,m_{lk}}\right)x_{l}^{m_{lk}}
$$ 
and 
$$\sum_{\mathbf p}B_{\mathbf n}^{\mathbf{q}}x^{\mathbf q}=\prod_{k=1}^r(\sum_{l=1}^r b_{lk}x_l)^{n_k}=\sum_{\substack{
N=(n_{lk})_{r\times r}\in \mathbb N^{r\times r}, \text{ each }n_{lk}\geq 0,\\
\sum_{l=1}^r n_{lk}=n_k,\ \\[2pt]
\sum_{k=1}^r n_{lk}=q_l\ 
}}
\left(\prod_{k=1}^{r}\frac{n_k!}{\prod_{l=1}^{r} n_{lk}!}\right)
\left(\prod_{l,k=1}^{r} b_{lk}^{\,n_{lk}}\right)x_{l}^{n_{lk}}
$$ 
with the related multi-indices $\mathbf p=(p_1,p_2,\cdots, p_r)$ and $\mathbf q=(q_1,q_2,\cdots, q_r)$.

 For the monomials $\mathbf m=(m_1,m_2,\cdots, m_r)$, $\mathbf n=(n_1,n_2,\cdots, n_r)$ and $\mathbf N=(N_1,N_2,\cdots, N_r)$, a linear functional $\varphi \in A^{*}=\operatorname{Hom}_{\mathbb K}(\mathbb K[x_1,x_2,\cdots,x_r],\mathbb K)$ lies in the \emph{Sweedler duality}
\[
  A^{\circ}
  \;=\;
  \bigl\{\varphi\in A^{*}\mid 
         \varphi\!\bigl(I\bigr)=0, \text{ $I=\sum_{\mathbf m}x^{\mathbf m}\mathbb K[x_1,x_2,\cdots,x_r]$ satisfying }\exists~ \mathbf N \text{ such that } x^{\mathbf{n}}\in I \text{ if } \exists~ 1\leq i\leq r, n_i-N_i\geq 0
     \bigr\}
\]
Hence $(A^{\circ},\Delta,\eta,\mu)$ is the BiHom-coalgebra induced by the Sweedler duality where the coordinate functionals span the maps $\mu$ and $\eta$
\[
  \mu(d_{\mathbf n})(x^{\mathbf m})=d_{\mathbf n}(\alpha(x^{\mathbf m})) =d_{\mathbf n}(\sum_{\mathbf p}A_{\mathbf m}^{\mathbf p}x^{\mathbf p})= \sum_{\mathbf p}A_{\mathbf m}^{\mathbf p}\delta_{\mathbf n,\mathbf p},
  \qquad \mathbf n=(n_1,n_2,\cdots,n_r),~\mathbf m=(m_1,m_2,\cdots,m_r),
\]
and 
\[
  \eta(d_{\mathbf n})(x^{\mathbf m})=d_{\mathbf n}(\beta(x^{\mathbf m})) =d_{\mathbf n}(\sum_{\mathbf q}B_{\mathbf m}^{\mathbf q}x^{\mathbf q})= \sum_{\mathbf q}B_{\mathbf m}^{\mathbf q}\delta_{\mathbf n,\mathbf q},
  \qquad \mathbf n=(n_1,n_2,\cdots,n_r),~\mathbf m=(m_1,m_2,\cdots,m_r).
\]
dualising the multiplication $\cdot\colon \mathbb K[x]\!\otimes\!\mathbb K[x]\to\mathbb K[x]$ yields a comultiplication
\[
  \Delta(d_{\mathbf n})
  \;=\;
  \sum_{\mathbf i+\mathbf j=\mathbf n} \mu(d_{\mathbf i})\otimes \eta(d_{\,\mathbf j}),
  \qquad \mathbf n=(n_1,n_2,\cdots,n_r).
\]

Specially, when $r=1$, it is the BiHom-(co)algebra over the univariant polynomial ring $\mathbb K[x]$ and $I=(x^{N})$, the coefficient matrices $A,B$ revert to numbers in $\mathbb K$; when $r=1$ and $\alpha=\beta$, the structure of Hom-(co)algebra in~\cite[Example 2.12]{SWZZ} holds.
\end{example}

 \section{Sweedler duality for BiHom-modules}\label{sect3}
In this section we recall BiHom-(co)modules and their morphisms (cf.~\cite[Definition~4.1]{DY} for the Hom case) and formulate the Sweedler dual construction in this context.
For a right BiHom-module $M$ over a BiHom-algebra $G$ and an ideal $J\subseteq G$, we denote by $M\cdot J$ the $\mathbb K$-subspace of $M$ spanned by all elements $m\cdot j$ with $m\in M$ and $j\in J$; equivalently, $M\cdot J$ is the product submodule usually written $MJ$.

    \begin{definition}
        Let $(G, \mu, \alpha,\beta)$ be a BiHom-algebra. A right BiHom-module of $G$ is a quadruple $(M, \rho, \kappa,\tau)$ where $M$ is a linear space, a linear map $\rho: M \otimes G \rightarrow M$ with $\rho(m\otimes g)=m\cdot g$ and two commutative linear maps $\kappa,\tau: M \rightarrow M$, that is, $\kappa\circ\tau=\tau\circ\kappa$, satisfying:
        \begin{align}
            \rho\circ(\rho\otimes\beta)&=\rho\circ(\kappa\otimes\mu),\quad \text{(BiHom-associativity)}\label{zd1}\\
        \rho\circ(\kappa\otimes\alpha)&=\kappa\circ\rho,\quad \quad \quad \quad
        \text{($\kappa$-multiplicativity)}\\
        \rho\circ(\tau\otimes\beta)&=\tau\circ\rho,\quad \quad \quad \quad
        \text{($\tau$-multiplicativity)}
        \end{align}
        
    \end{definition}

    We will refer to a right BiHom-module $(M,\rho,\kappa,\tau)$ as $M$.
     Let $(M,\rho,\kappa,\tau)$ and $(N,\tilde{\rho},\tilde{\kappa},\tilde{\tau})$ be two right BiHom-modules of $G$. A map $\sigma: M \rightarrow N$ is said to be a morphism of right BiHom-modules when the following conditions hold:
     \begin{align}\label{s7}
    \sigma\circ \rho = \tilde{\rho} \circ (\sigma\otimes\operatorname{Id}_G), \quad \quad
    \tilde{\kappa} \circ \sigma = \sigma \circ \kappa, \quad \quad
    \tilde{\tau} \circ \sigma = \sigma \circ \tau.
       \end{align}

\begin{definition}
    Let $(C, \Delta, \psi,\phi)$ be a BiHom-coalgebra. A right BiHom-comodule of $C$ is a triple $(A, \gamma, \omega,\theta)$ where $A$ is a linear space, a linear map $\gamma: A \rightarrow A\otimes C$ with $\gamma(a)=\sum_{(a)}a_{(0)}\otimes a_{(1)}$ and two commutative linear maps $\omega,\theta: C \rightarrow C$, i.e., $\omega\circ\theta=\theta\circ\omega$, satisfying:
    \begin{align*}
     (\gamma\otimes\psi)\circ\gamma&=(\theta\otimes\Delta)\circ\gamma,\quad \quad \,\,\text{(BiHom-coassociativity)}\\
    (\omega\otimes\psi)\circ\gamma&=\gamma\circ\omega,\quad \quad \quad \quad \quad \text{($\omega$-comultiplicativity)} \\
    (\theta\otimes\phi)\circ\gamma&=\gamma\circ\theta.\quad \quad \quad \quad \quad \text{($\theta$-comultiplicativity)} 
    \end{align*}
    
\end{definition}

We write $(A,\gamma,\omega,\theta)$ for a right BiHom-comodule over a BiHom-coalgebra $C$.
Given two right BiHom-comodules $(A,\gamma,\omega,\theta)$ and $(A',\gamma',\omega',\theta')$ over $C$,
a linear map $f:A\to A'$ is called a morphism of right BiHom-comodules if
\begin{align*}
\gamma'\circ f = (f\otimes \mathrm{Id}_C)\circ \gamma,\quad
\omega'\circ f = f\circ \omega,\quad
\theta'\circ f = f\circ \theta.
\end{align*}

\begin{definition}
    The \textit{Sweedler duality} for the right BiHom-module $M$ based on the ordinary duality $M^*$ is defined by
    $$M^{\circ} := \{ f \in M^{*} \mid f(M\cdot I) = 0, I \text{ is a  finite-codimensional ideal of  G  } \}.$$
\end{definition}

Now we give the isomorphism relation in dual spaces (see~\cite{LGL})  and
the structure of BiHom-comodule induced by the universal duality as follows.
\begin{lemma}\label{iso}
    Let $U$ and $V$ be two linear spaces. Denote $\rho:U^*\otimes V^*\rightarrow (U\otimes V)^*$ with corresponding rule $\rho(f\otimes g)(x\otimes y)=f(x)g(y)$, for any $f\in U^*$, $g\in V^*$, $x\in U$ and $y\in V$. If $\operatorname{dim} U<+\infty$ or $\operatorname{dim} V<+\infty$, $\rho$ is an isomorphism of linear spaces.
\end{lemma}

\begin{lemma} \label{q}
    Let $(G, \mu, \alpha,\beta)$ be a finite-dimensional BiHom-algebra and $(M, \rho, \kappa,\tau)$ be a right BiHom-module of $G$. Then, $(M^{*}, \rho^{*},\tau^{*}, \kappa^{*})$ is a right BiHom-comodule of a BiHom-coalgebra $(G^{*}, \mu^{*}, \beta^{*},\alpha^{*})$, where $\rho^{*}: M^{*} \rightarrow M^{*} \otimes G^{*}$ and $\kappa^{*},\tau^*:M^*\rightarrow M^*$.
\end{lemma}

\begin{proof}
    By Lemma~\ref{s}, we can find $(G^{*}, \mu^{*},\beta^*, \alpha^{*})$ is a BiHom-coalgebra. 
    For any $m^*\in M^*$, $m\in M$ and $g\in G$, we have
    \begin{align*}
        \langle \rho^*\circ\kappa^*(m^*), m\otimes g\rangle&=\langle m^*, \kappa\circ\rho(m\otimes g)\rangle\\
        &=\langle m^*, \rho\circ(\kappa\otimes \alpha)(m\otimes g)\rangle\\
        &=\langle \rho^*(m^*), \kappa(m)\otimes \alpha(g)\rangle\\
        &=\sum_{(m^*)}\langle m^*_{(0)}\otimes m^*_{(1)}, \kappa(m)\otimes \alpha(g)\rangle\\
        &=\sum_{(m^*)}\langle \kappa^*(m^*_{(0)}), m\rangle\langle \alpha^*(m^*_{(1)}), g\rangle\\
        &=\sum_{(m^*)}\langle (\kappa^*\otimes \alpha^*)(m^*_{(0)}\otimes m^*_{(1)}), m\otimes g\rangle\\
        &=\langle (\kappa^*\otimes \alpha^*)\circ\rho^*(m^*),  m\otimes g\rangle.
    \end{align*}
    Therefore we have $\rho^*\circ\kappa^*=(\kappa^*\otimes \alpha^*)\circ\rho^*$. Similarly, $\rho^*\circ\tau^*=(\tau^*\otimes \beta^*)\circ\rho^*$.

    Now it is sufficient for us to check that $(\rho^*\otimes\beta^*)\circ\rho^*=(\kappa^*\otimes\mu^*)\circ\rho^*$. For any $a^*\in M^*$ and $x\in M, y,z\in G$,
    \begin{align*}
        \langle(\rho^*\otimes\beta^*)\circ\rho^*(a^*), x\otimes y\otimes z\rangle&=\langle(\rho^*\otimes\beta^*)(\sum_{(a^*)}a^*_{(0)}\otimes a^*_{(1)}), x\otimes y\otimes z\rangle\\
        &=\sum_{(a^*)}\langle \rho^*(a^*_{(0)}), x\otimes y \rangle\langle \beta^*(a^*_{(1)}), z \rangle\\
        &=\sum_{(a^*)}\langle a^*_{(0)}, \rho(x\otimes y) \rangle\langle a^*_{(1)}, \beta(z)\rangle\\
        &=\langle \rho^*(a^*), (\rho\otimes\beta)(x\otimes y\otimes z)\rangle\\
        &=\langle a^*, \rho\circ(\rho\otimes\beta)(x\otimes y\otimes z)\rangle.
    \end{align*}
    Similarly, we have $\langle(\kappa^*\otimes\mu^*)\circ\rho^*(a^*), x\otimes y\otimes z\rangle=\langle a^*, \rho\circ(\kappa\otimes\mu)(x\otimes y\otimes z)\rangle$. 
    
    According to the equation~\eqref{zd1}, It can be shown that $(\rho^*\otimes\beta^*)\circ\rho^*=(\kappa^*\otimes\mu^*)\circ\rho^*$. This completes the proof of the lemma.
\end{proof}

We now pass to the module setting and start by recording that the Sweedler duality for a right BiHom-module is a linear subspace.

\begin{theorem}\label{c'}
Let $G$ be a BiHom-algebra and $M$ a right BiHom-module over $G$.
Then $M^{\circ}$ is a $\mathbb K$-subspace of $M^{*}$.
\end{theorem}

\begin{proof}
Clearly $0\in M^{\circ}$ and $M^{\circ}$ is stable under scalar multiplication.
Let $f,g\in M^{\circ}$, and choose finite-codimensional ideals $I,J\subseteq G$ with
$f(M\cdot I)=0$ and $g(M\cdot J)=0$.
Set $K:=I\cap J$. Then $K$ is a finite-codimensional ideal of $G$ and $M\cdot K\subseteq M\cdot I$ and $M\cdot K\subseteq M\cdot J$.
Hence $(f+g)(M\cdot K)=0$, so $f+g\in M^{\circ}$.
\end{proof}

\begin{theorem}\label{f2}
Let $(G,\mu,\alpha,\beta)$ be a BiHom-algebra with $\beta$ surjective, and let
$(M,\rho,\kappa,\tau)$ be a right BiHom-module over $G$.
Define
\[
G^{\circ}\subseteq G^{*},\qquad M^{\circ}\subseteq M^{*}
\]
to be the Sweedler duals, and set
\[
\rho^{\circ}:=\rho^{*}\big|_{M^{\circ}},\qquad
\kappa^{\circ}:=\kappa^{*}\big|_{M^{\circ}},\qquad
\tau^{\circ}:=\tau^{*}\big|_{M^{\circ}},
\]
where $\rho^{*}:M^{*}\to (M\otimes G)^{*}$ is the linear dual of $\rho$.
Then $\rho^{\circ}$ takes values in $M^{\circ}\otimes G^{\circ}$, and
\[
(M^{\circ},\rho^{\circ},\tau^{\circ},\kappa^{\circ})
\]
is a right BiHom-comodule over the Sweedler dual BiHom-coalgebra
$(G^{\circ},\Delta,\beta^{\circ},\alpha^{\circ})$.
\end{theorem}

\begin{proof}
Noticing that $f\in M^{\circ}$, we aim to verify that $\rho^{\circ}(f)\in M^{\circ}\otimes G^{\circ}$. There exists a finite-codimensional ideal $J$ of $G$ such that $f(M\cdot J) = 0$. Since $\beta$ is the epimorphism map of $G$, $\kappa(M)\subseteq M$, $\tau(M)\subseteq M$ and $JG\subseteq J$, we have
    \begin{align*}
        \rho^{\circ}(f)(M\cdot J \otimes  G + M \otimes J) &=
         \rho^{\circ}(f)(M\cdot J \otimes \beta(G) + M \otimes J) \\
        &=f\circ\rho(M\cdot J \otimes \beta(G) + M \otimes J) \\
        &= f(\rho(M\cdot J \otimes \beta(G)) + \rho(M \otimes J)) \\
        &= f(\rho(\kappa(M) \otimes JG)+\rho(M\otimes J))\\
        &= 0.
    \end{align*}
    Denote $\pi : M \otimes G \rightarrow M \otimes G / (M\cdot J \otimes G + M \otimes J)$ and $\operatorname{Ker}(\pi) = M\cdot J \otimes G + M \otimes J \subseteq \operatorname{Ker}(f\circ \rho).$
    Using Lemma~\ref{zz}, we can induce a unique morphism $\overline{f\rho} : M \otimes G / (M\cdot J \otimes G + M \otimes J) \rightarrow \mathbb K$ satisfying $\overline{f\rho}\circ \pi=f\circ\rho$, that is to say,
    \begin{align*}
    \overline{f\rho} \in \left( M \otimes G / (M\cdot J \otimes G + M \otimes J) \right)^*.
    \end{align*}

    According to the Lemma~\ref{iso}, we can get $M \otimes G / (M\cdot J \otimes G + M \otimes J)\cong (M / M\cdot J) \otimes (G / J)$ with finite-dimensional linear space $G / J$, so we have
\begin{align*}
    \overline{f \rho} \in (M / M\cdot J)^* \otimes (G / J)^*.
\end{align*}

    Denote two canonical maps $\pi_1 : M \rightarrow M/M\cdot J$ and $\pi_2 : G \rightarrow G/J$. And then we denote $\overline{f \rho} = \sum^n_{i,j=1}k_{ij}\overline{d_i}^* \otimes \overline{e_j}^*$ where $\overline{d_i}^* \in (M/M\cdot J)^*$, $\overline{e_j}^* \in (G/J)^*$ are the dual base elements. After a similar discussion as in Theorem~\ref{f}, for any $x\in M$ and $y\in G$, we can get that
    $$
    \langle f\circ\rho, x\otimes y\rangle=\sum_{i,j=1}^n k_{ij}\langle\overline{d_i}^*,\pi_1(x)\rangle\langle\overline{e_j}^*,\pi_2(y)\rangle.
    $$
    Denote $d_i^*=\overline{d_i}^*\circ\pi_1$ and $e_j^*=\overline{e_j}^*\circ\pi_2$. Since $d_i^*(M\cdot J)=0$ and $\operatorname{Ker}(e_j^*)\supseteq J$, we have $d_i^*\in M^{\circ}$ and $e_j^*\in G^{\circ}$. Thus, we have
    $$
    \rho^{\circ}(f)=f\circ\rho=\sum_{i,j=1}^n k_{ij} d_i^*\otimes e_j^*.
    $$  Consequently, $\rho^{\circ}(f) \in M^{\circ} \otimes G^{\circ}$.\par
    For any $m\in M^{\circ}$, there exists a finite-codimensional ideal $J$ of $G$ such that $m(M\cdot J)=0$. Then we have
    \begin{align*}
        \langle \kappa^{\circ}(m), M\cdot J\rangle=\langle m ,\kappa(M\cdot J)\rangle=0,
    \end{align*}
    since $\kappa(M\cdot J)=\kappa(M)\cdot \alpha(J)\subseteq M\cdot J$. Then we have $\kappa^{\circ}(m)(M\cdot J)=0$, which implies $\kappa^{\circ}(m)\in M^{\circ}$. Thus, $\kappa^{\circ}(M^{\circ})\subseteq M^{\circ}$. Similarly, $\tau^{\circ}(M^{\circ})\subseteq M^{\circ}$.
    Combining Theorem~\ref{c}, Theorem~\ref{c'}, and Lemma~\ref{q} together, we deduce that $\rho^{\circ}\circ\kappa^{\circ}=(\kappa^{\circ}\otimes \alpha^{\circ})\circ\rho^{\circ}$, $\rho^{\circ}\circ\tau^{\circ}=(\tau^{\circ}\otimes \beta^{\circ})\circ\rho^{\circ}$ and $(\rho^{\circ}\otimes\beta^{\circ})\circ\rho^{\circ}=(\kappa^{\circ}\otimes\Delta)\circ\rho^{\circ}$. Hence, $M^{\circ}$ is a right BiHom-comodule of $G^{\circ}$.
\end{proof}

\begin{remark}
The surjectivity of the twisting map $\beta:G\to G$ is used essentially in the proof of Theorem~\ref{f2}.
Indeed, $\beta(G)=G$ implies
\[
M\cdot J\otimes G \;=\; M\cdot J\otimes \beta(G),
\]
and therefore, for $f\in M^{\circ}$ annihilating $M\cdot J$, we obtain
\[
\rho^{\circ}(f)\bigl(M\cdot J\otimes G + M\otimes J\bigr)=0.
\]
This vanishing property is essential. It ensures that $\rho^{\circ}(f)$ descends to the quotient tensor product 
\[
(M/M\cdot J)\otimes (G/J)
\]
which in turn implies that $M^{\circ}\otimes G^{\circ}$.
\end{remark}

Having constructed the induced right $G^{\circ}$-BiHom-comodule $M^{\circ}$ in Theorem~\ref{f2}, we show that the construction is compatible with morphisms, leading to the following result.

\begin{theorem}\label{f3}
Let $(G,\mu,\alpha,\beta)$ be a BiHom-algebra with $\beta$ surjective, and let
$(M,\rho,\kappa,\tau)$ and $(N,\rho',\kappa',\tau')$ be right BiHom-modules over $G$.
Assume that $\sigma:M\to N$ is a morphism of right BiHom-modules.
Then the restricted dual map
\[
\sigma^{\circ}:=\sigma^{*}\big|_{N^{\circ}}:N^{\circ}\longrightarrow M^{\circ}
\]
is a morphism of right BiHom-comodules over the Sweedler dual coalgebra $G^{\circ}$.
\end{theorem}

\begin{proof}
    For all $f \in N^{\circ}$, let $J$ be a finite-codimensional ideal of $G$ such that $f(N\cdot J) = 0$. By the properties of the morphism of the right Hom-modules, we have
    \begin{align*}
    \sigma^{\circ}(f)(M\cdot J) &= \sigma^{\circ}(f)\circ\rho(M \otimes J) \\
    &= f(\sigma\circ \rho(M \otimes J)) \\
    &= f(\rho' \circ (\sigma \otimes \operatorname{Id}_G)(M \otimes J)) \\
    &= f(\sigma(M)\cdot J)\\
    &= 0,
    \end{align*}
    since $\sigma(M)\subseteq N$. Hence $\sigma^{\circ}(N^{\circ}) \subseteq M^{\circ}$.
    
    For any $b^*\in N^{\circ}$ and $x\in M, y\in G$, we have
    \begin{align*}
        \langle(\sigma^{\circ}\otimes\operatorname{Id}_{G^{\circ}})\circ\rho'^{\circ}(b^*), x\otimes y\rangle&=\sum_{(b^*)}\langle(\sigma^{\circ}\otimes\operatorname{Id}_{G^{\circ}})(b^*_{(0)}\otimes b^*_{(1)}), x\otimes y\rangle\\
        &=\sum_{(b^*)}\langle \sigma^{\circ}( b^*_{(0)}), x\rangle\langle b^*_{(1)}, y\rangle\\
        &=\langle \rho'^{\circ}(b^*),  (\sigma\otimes\operatorname{Id}_G)(x\otimes y)\rangle\\
        &=\langle b^*, \rho'\circ(\sigma\otimes\operatorname{Id}_G)(x\otimes y)\rangle\\
        &=\langle b^*, \sigma\circ\rho(x\otimes y)\rangle\\
        &=\langle\rho^{\circ}\circ\sigma^{\circ}(b^*), x\otimes y\rangle.
    \end{align*}
    Thus we obtain $(\sigma^{\circ}\otimes\operatorname{Id}_{G^{\circ}})\circ\rho'^{\circ}=\rho^{\circ}\circ\sigma^{\circ}$.
    And then for all $c^*\in N^{\circ}$ and $x\in M$,
    \begin{align*}
        \langle \kappa^{\circ}\circ\sigma^{\circ}(c^*), x\rangle
        =\langle c^*, \sigma\circ\kappa(x)\rangle
        =\langle c^*, \kappa'\circ\sigma(x)\rangle
        =\langle \sigma^{\circ}\circ\kappa'^{\circ}(c^*), x\rangle.
    \end{align*}
    Thus it follows that  $\kappa^{\circ}\circ\sigma^{\circ}=\sigma^{\circ}\circ\kappa'^{\circ}$. Similarly, $\tau^{\circ}\circ\sigma^{\circ}=\sigma^{\circ}\circ\tau'^{\circ}$. 
\end{proof}

\bibliographystyle{plain}

\end{document}